\documentclass{article}

\usepackage{graphicx}
\usepackage{indentfirst}
\usepackage{amsmath,amsfonts,amsthm,amssymb}
\usepackage{mathrsfs}
\usepackage{amscd}
\usepackage{color}
\usepackage{hyperref}

\def\co{\colon\thinspace}
\DeclareMathAlphabet{\mathsfsl}{OT1}{cmss}{m}{sl}

\newcommand{\spin}{\mathrm{Spin}^c}
\newcommand{\bZ}{{\mathbb{Z}}}
\newcommand{\del}{{\partial}}

\newtheorem{thm}{Theorem}[section]
\newtheorem{lem}[thm]{Lemma}
\newtheorem{cor}[thm]{Corollary}
\newtheorem{prop}[thm]{Proposition}

\theoremstyle{definition}
\newtheorem{defn}[thm]{Definition}

\newtheorem{rem}[thm]{Remark}

\newtheorem{example}[thm]{Example}

\begin{document}

\title{Non-simple genus minimizers in lens spaces}

\author{{\Large Joshua Evan Greene}\\{\normalsize Department of Mathematics, Boston College}\\
{\normalsize Carney Hall, Chestnut Hill, MA
02467}\\{\small\it Email\/:\quad\rm joshua.greene@bc.edu}
\vspace{.1in}\\
{\Large Yi Ni}\\{\normalsize Department of Mathematics, Caltech, MC 253-37}\\
{\normalsize 1200 E California Blvd, Pasadena, CA
91125}\\{\small\it Email\/:\quad\rm yini@caltech.edu}}

\date{}
\maketitle

\begin{abstract}
Given a one-dimensional homology class in a lens space, a question related to the Berge conjecture on lens space surgeries is to determine all knots realizing the minimal rational genus of all knots in this homology class. It is known that simple knots are rational genus minimizers. In this paper, we construct many non-simple genus minimizers. This negatively answers a question of Rasmussen.
\end{abstract}

\section{Introduction}

Heegaard Floer homology, introduced by Ozsv\'ath and Szab\'o \cite{OSzAnn1}, has been very successful in the study of low-dimensional topology. One important property of Heegaard Floer homology is that it detects Thurston norm of $3$--manifolds \cite{OSzGenus}. For a closed oriented $3$--manifold $Y$, its Thurston norm is always zero on the torsion subgroup of $H_1(Y;\mathbb Z)$. However, there is another kind of ``norm'' function we can define on the torsion subgroup of $H_1(Y;\mathbb Z)$. To define it, let us first recall the rational genus of a rationally null-homologous knot $K\subset Y$ defined by Calegari and Gordon \cite{CG}.

Suppose that $K$ is a rationally null-homologous oriented knot in $Y$, and $\nu(K)$ is a tubular neighborhood of $K$. A properly embedded, oriented, connected surface $F\subset Y\backslash\overset{\circ}{\nu}(K)$ is called a {\it rational Seifert surface} for $K$ if $\partial F$ consists of coherently oriented parallel curves on $\partial\nu(K)$ and the orientation of $\partial F$ is coherent with the orientation of $K$. The {\it rational genus} of $K$ is defined to be
$$g_r(K)=\min_{F}\frac{\max\{0,-\chi(F)\}}{2|[\mu]\cdot[\partial F]|},$$
where $F$ runs over all the rational Seifert surfaces for $K$ and $\mu\subset\partial\nu(K)$ is the meridian of $K$.

The rational genus is a natural generalization of the genus of null-homologous knots. Moreover, given a torsion class in $H_1(Y)$, we can consider the minimal rational genus over all knots in this torsion class. More precisely,
given $a\in\mathrm{Tors} H_1(Y)$, let 
$$\Theta(a)=\min_{K\subset Y,\:[K]=a}2g_r(K).$$
This function $\Theta$ was introduced by Turaev \cite{TuFunc} in a slightly different form. Turaev regarded $\Theta$ as an analogue of Thurston norm \cite{Th}, in the sense that it measures the minimal normalized Euler characteristic of a ``folded surface'' representing a given class in $H_2(Y;\mathbb Q/\mathbb Z)$.

Recall that a rational homology $3$--sphere $Y$ is an {\it L-space} if $\mathrm{rank}\widehat{HF}(Y)=|H_1(Y;\mathbb Z)|$. A rationally null-homologous knot $K$ in a $3$--manifold $Y$ is {\it Floer simple} if 
$\mathrm{rank}\widehat{HFK}(Y,K)=\mathrm{rank}\widehat{HF}(Y)$. In \cite{NiWu}, it is shown that Floer simple knots in L-spaces have the smallest rational genus among all knots in the same homology class.

An important class of L-spaces is lens spaces.
Hedden \cite{HedBerge} and Rasmussen \cite{RasBerge} observed that for any one-dimensional homology class in a lens space, there exists a knot in this homology class which is Floer simple. Let $U_1\cup U_2$ be a genus one Heegaard splitting of a lens space $L(p,q)$, and let $D_1,D_2$ be meridian disks in $U_1,U_2$ such that $\partial D_1\cap\partial D_2$ consists of exactly $p$ points. A knot in $L(p,q)$ is called {\it simple} if it is either the unknot or the union of two arcs $\gamma_1\subset D_1$ and $\gamma_2\subset D_2$. Up to isotopy there is exactly one simple knot in each homology class, and every simple knot is Floer simple.  Hence simple knots in lens spaces are genus minimizers in their homology classes, which answers a question of Rasmussen \cite{RasBerge} affirmatively.

Rasmussen \cite{RasBerge} also asked whether simple knots are the unique genus minimizers in their homology classes, and if they are not, what the genus minimizers are. Baker \cite{Baker} showed that if the rational genus of a knot in $L(p,q)$ is less than $\frac14$, and the minimal rational Seifert surface has only one boundary component, then the knot is a one-bridge knot. In this paper, we answer Rasmussen's question {\em negatively} in some cases, and in a very strong sense.

\begin{thm}\label{thm:NonUnique}
There exist infinitely many triples $(p,q,a)$ such that the simple knot is not the unique genus minimizer in the homology class $a\in H_1(L(p,q))$.  
Moreover, there exist infinitely many triples $(p,q,a)$ such that there exist infinitely many non-isotopic genus minimizers in the homology class $a \in H_1(L(p,q))$.
\end{thm}
Baker gave an independent proof of Theorem \ref{thm:NonUnique} using the methods of \cite{BGL} after learning about some of our examples \cite{Bakercorr}.

Figure~\ref{fig:FirstExample} displays the ``simplest" non-simple genus minimizer we found. The knot $K_0$ is the $(1,2)$--cable of the $(1,2)$--torus knot in $L(8,1)$. The simple knot in this homology class is the $(1,4)$--torus knot.  These two knots are not isotopic because their complements are not homeomorphic. 
The knot $K_0$ recurs throughout the various classes of examples that we study.

\begin{figure}[ht]
\centering
\includegraphics[height=120pt]{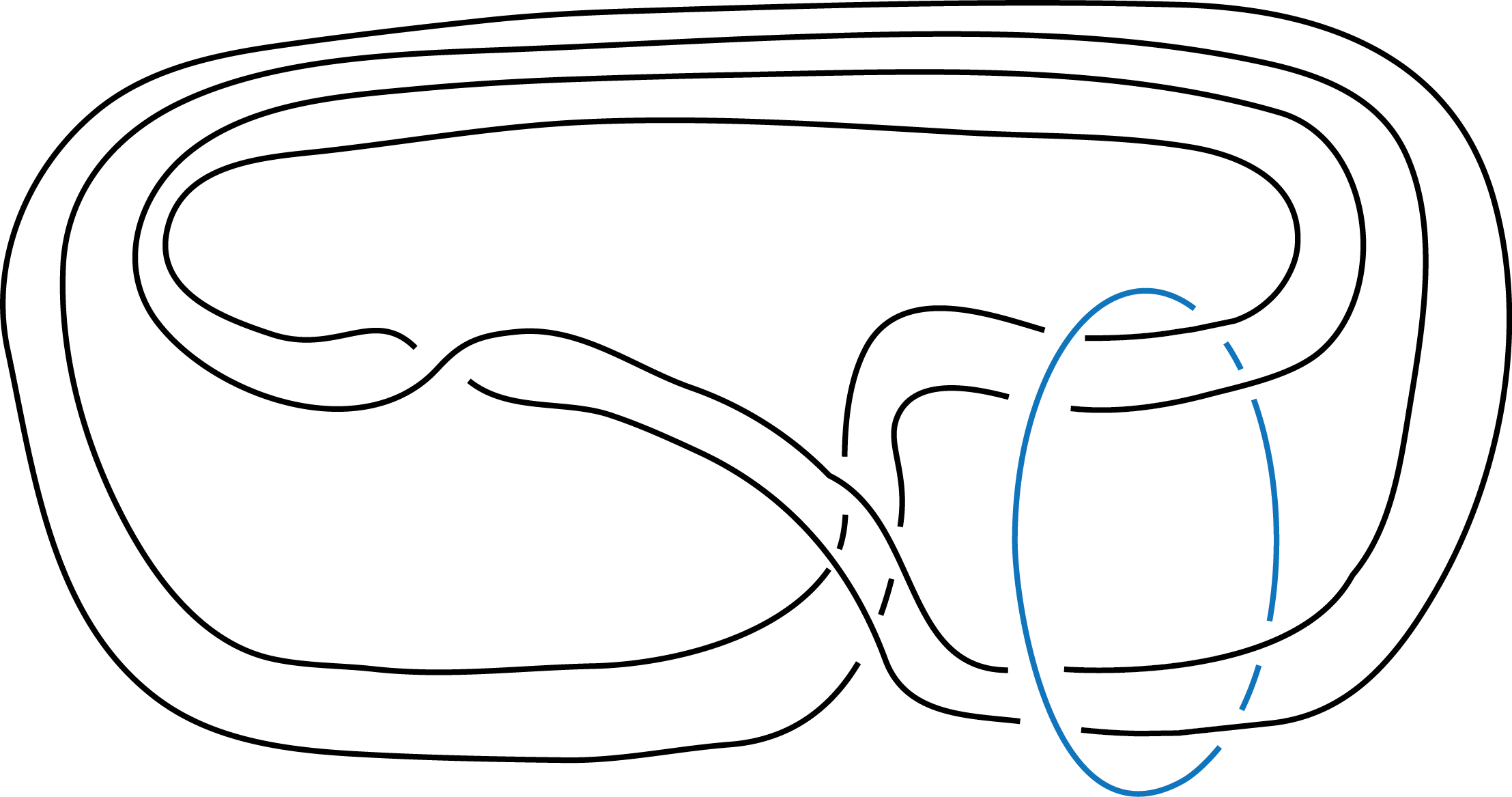}
\put(0,50){$K_0$}
\put(-45,-1){$\textcolor{blue}{8}$}
\caption{\label{fig:FirstExample}The ``simplest" non-simple genus minimizer in $L(8,1)$. The ambient manifold is obtained by $8$--surgery on a component of the link, and the other component is a genus minimizer in $L(8,1)$.}
\end{figure}

The examples we construct all have large rational genus, and they do not appear to be Floer simple knots.  For instance, $K_0$ has rational genus $\frac12$, and we confirmed that it is not Floer simple.  It is possible that simple knots with small rational genus are the unique genus minimizers in their homology classes.  For example, this could be the case for knots with rational genus less than $\frac12$, which include the dual Berge knots \cite{HedBerge,RasBerge}.  It is also possible that simple knots are the only Floer simple knots in lens spaces. 

The lens space $L(2k,q)$ contains a (geometrically) incompressible closed surface $\Pi$, which is necessarily nonorientable.
Rubinstein \cite{RubinOS} proved that $\Pi$ is unique up to isotopy, and its genus $N(2k,q)$ can be computed by \cite{BW} (see also the end of Section \ref{sect:TorsCurve}). Recently, Johnson \cite{Johnson} proved that any nonorientable surface in 
$L(2k,q)$ with a given genus $\ge N(2k,q)$ is unique up to isotopy.  Rubinstein's result characterizes the genus-minimizing knots in the homology class $k\in H_1(L(2k,q))$.

\begin{thm}\label{thm:Order2Minimizer}
In the lens space $L(2k,q)$, a knot $K$ in the homology class $k$ is a genus minimizer if and only if $K$ is isotopic to a curve $C$ on the incompressible surface $\Pi$ in $L(2k,q)$ so that $\Pi-C$ is orientable. In particular, there is a unique genus minimizer in this homology class if $N(2k,q)\le3$, i.e., $\Theta(k) \le \frac14$.
\end{thm}

\noindent  By contrast, the smallest pair $(2k,q)$ for which $N(2k,q) \ge 4$ is $(8,1)$, in which case we have remarked that the knot $K_0 \subset L(8,1)$ is a non-simple genus minimizer.


The paper is organized as follows. In Section~\ref{sect:Cable}, we construct a family of non-simple rational genus minimizers via the cabling construction. In Section~\ref{sect:TorsCurve}, we describe all the rational genus minimizers in order 2 homology classes.  In Section~\ref{sect:Stab}, we get more examples (most of them hyperbolic) by stabilizing the example in Figure~\ref{fig:FirstExample}. In Section~\ref{sect:InfiniteMany}, we show that some order 2 homology classes contain infinitely many non-isotopic genus minimizers by using the annulus twist construction \cite{Osoinach}.

\vspace{5pt}\noindent{\bf Acknowledgements.}\quad We thank Ian Agol, Ken Baker, Josh Batson, Cameron Gordon, Adam Levine, Danny Ruberman, and Sucharit Sarkar for helpful discussions.  We especially thank Eli Grigsby for providing her computer program to compute knot Floer homology of knots in lens spaces.  We also thank the Simons Center for Geometry and Physics and the organizers of the workshop ``Symplectic and Low Dimensional Topologies in Interaction'' for providing an ideal place for us to collaborate.  The first author was partially supported by NSF grant number DMS-1207812.  The second author was partially supported by NSF grant number DMS-1103976 and an Alfred P. Sloan Research Fellowship. 



\section{Examples of cable knots}\label{sect:Cable}

From now on we assume that $p>q\ge1$ and $\gcd(p,q)=1$. We choose the following presentation for the lens space $L(p,q)$. Take a positive Hopf link $L_1\cup L_2\subset S^3$. $L(p,q)$ is obtained by $\frac pq$--surgery on $L_2$. Let $U_1$ be a tubular neighborhood of $L_1$, $T=\partial U_1$, and $U_2=\overline{L(p,q)-U_1}$.  Thus, $U_1\cup_T U_2$ is the (unique) genus one Heegaard splitting of $L(p,q)$. 

\begin{defn}\label{defn:Simple}
We recall the definition of simple knots from \cite{RasBerge}. Let $(T^2,\alpha,\beta)$ be a genus one Heegaard diagram for $L(p,q)$, where $\alpha$ is the meridian of $U_1$, $\beta$ is the meridian of $U_2$, and $|\alpha \cap \beta| = p$.  Orient $\alpha$ so that its linking number with $L_1$ is positive.  Let $A\subset U_1$ be the meridian disk bounded by $\alpha$ and $B\subset U_2$ be the meridian disk bounded by $\beta$. Pick an intersection point $x_0 \in \alpha \cap \beta$ and label the remaining intersection points by $x_1,\dots,x_{p-1}$ in order of their appearance around $\alpha$, using the orientation on $\alpha$.  For each $a\in\{1,\dots,p-1\}$, the simple knot $K(p,q,a)$ is the oriented knot that is the union of an arc joining $x_0$ to $x_a$ in $A$ with an arc joining $x_a$ to $x_0$ in $B$.  Note that the isotopy type of $K(p,q,a)$ is independent of the choice of $x_0$, due to the symmetry of the Heegaard diagram.
\end{defn}

Let $[L_1]=1\in H_1(L(p,q))\cong\mathbb Z/p\mathbb Z$.
Define $$\varphi_q\co\{0,1,\dots,p-1\}\to \{0,1,\dots,p-1\}$$ by 
requiring  $q\varphi_q(a)\equiv a\pmod p$. 
Then the knot $K(p,q,a)$ represents the homology class $\varphi_q(a)$. In fact,
if we push $K(p,q,a)$ into $U_1$, we get an index--$\varphi_q(a)$ closed braid in $U_1$.

\begin{defn}
For a simple, closed, essential curve on $T$, we define its slope $\frac rs\in\mathbb Q\cup\{\infty\}$ with respect to the canonical longitude of $L_1$.
This curve is then called the {\it $(r,s)$--torus knot} in $L(p,q)$.
\end{defn}

In general, for any knot $K\subset U_1$, there is a canonical longitude of $K$ which is null-homologous 
in $S^3-\nu(K)\supset U_1-\nu(K)$. As usual in Dehn surgery, we parametrize the set of slopes on $K$ by $\mathbb Q\cup \{ \infty \}$ so that the canonical longitude is $0$.

\begin{lem}\label{lem:TorusCondition}
When $k<\frac pq+1$, $K(p,q,qk)$ is the $(1,k)$--torus knot.  Thus, the $(1,k)$--torus knot represents the homology class $k\in H_1(L(p,q))$.
\end{lem}
\begin{proof}
As in Definition~\ref{defn:Simple}, $K(p,q,qk)$ is the union of two arcs. We can push these two arcs into the Heegaard torus $T^2$ such that one arc is contained in $\alpha$ and the other arc is contained in $\beta$. When $k< p/q+1$, the two arcs can be chosen so that their interiors are disjoint on $T^2$. So $K(p,q,qk)$ is a torus knot, and it is straightforward to see that it is the $(1,k)$--torus knot.
\end{proof}

The following lemma is elementary.

\begin{lem}\label{lem:BoundarySlope}
Suppose $K\subset U_1$ is a knot with winding number $w$. Let $\mu\subset\partial\nu(K)$ be the meridian of $K$, $\lambda\subset \partial\nu(K)$ the canonical longitude of $K$, and
$$
\iota_*\co H_1(\partial\nu(K))\to H_1(L(p,q)-\nu(K))
$$
the map induced by the inclusion.  Then $\ker\iota_*$ is generated by  $(w^2q/d)[\mu]+(p/d)[\lambda]$,  where $d = \gcd(w,p)$.
\end{lem}
\begin{proof}
Let $\mathcal M\subset \partial\nu(L_2)$ be the meridian of $L_2$ and $\mathcal L\subset \partial\nu(L_2)$ the canonical longitude of  $L_2$. The homology group $H_1(S^3-\nu(K\cup L_2))$ has a presentation
$$\langle[\mu],[\lambda],[\mathcal M],[\mathcal L]\:|\:w[\mu]=[\mathcal L], [\lambda]=w[\mathcal M]\rangle.$$
After $\frac pq$--surgery on $L_2$, we obtain a new relation $p[\mathcal M]+q[\mathcal L]=0$. So we get the following presentation matrix for $H_1(L(p,q)-\nu(K))$:
$$\begin{pmatrix}
w &0 &0 &-1\\
0 &1 &-w &0\\
0 &0 &p &q
\end{pmatrix}.
$$
Since $\gcd(w,p)=d$, there exists a pair of integers $r,s$ such that $rw+sp=d$. We multiply the above matrix on the left by
$$\begin{pmatrix}
1 &0 &0\\
-sq &r &-s\\
wq/d &p/d &w/d
\end{pmatrix}
\in \mathrm{SL}(3,\mathbb Z)$$
to get a new presentation matrix
$$
\begin{pmatrix}
w &0 &0 &-1\\
-sqw &r &-d &0\\
w^2q/d &p/d &0 &0
\end{pmatrix}.
$$
The conclusion follows from this presentation matrix.
\end{proof}

\begin{defn}
Given a compact surface $F$, define its {\it norm} $$\chi_-(F)=\sum_i\max\{-\chi(F_i),0\},$$ where the sum is taken over all components $F_i$ of $F$. Given $A\in H_2(M,\partial M)$ for a $3$--manifold $M$,
the {\it Thurston norm} of $A$ is defined to be
$$\chi_-(A)=\min\big\{\chi_-(F)\big|(F,\partial F)\subset(M,\partial M),\quad[F,\partial F]=A\big\}.$$
\end{defn}

The Thurston norm of a graph manifold is easily determined by the following lemma.

\begin{lem}\label{lem:GraphNorm}
Suppose that $M$ is a graph manifold with JSJ pieces $M_1,\dots,M_n$. Suppose that each $M_i$ has an oriented Seifert fibration over an oriented orbifold $\mathcal B_i$ of orbifold Euler characteristic $\chi_{\mathrm{orb}}(\mathcal B_i)<0$, and $C_i\subset M_i$ is a regular fiber.  Given a homology class $A\in H_2(M)$, we have 
$$
\chi_-(A)=\sum_i|(A\cdot[C_i]) \, \chi_{\mathrm{orb}} (\mathcal B_i)|.
$$
Moreover, if $A\cdot[C_i]\ne0$ for each $i$, then $A$ is represented by the fiber of a fibration of $M$ over $S^1$.
\end{lem}
\begin{proof}
Let $F$ be a (possibly disconnected) Thurston norm-minimizing surface that represents $A$. We may isotope $F$ such that each $F_i=F\cap M_i$ is an incompressible surface. It follows from the classification of incompressible surfaces in Seifert fibered spaces that $F_i$ consists of vertical annuli and tori if $[F_i]\cdot[C_i]=0$, and $F_i$ is the fiber of a fibration of $M_i$ over $S^1$ if $[F_i]\cdot[C_i]\ne0$ \cite[Theorem~6.34]{Jaco}.

If $F_i$ consists of vertical annuli and tori, then $F_i$ has no contribution to the norm of $F$. If $F_i$ is the fiber of a fibration, then $F_i$ is isotopic to a surface transverse to all Seifert fibers of $M_i$.  Then the Seifert fibration $M_i\to\mathcal B_i$ induces a covering map $F_i\to\mathcal B_i$ whose degree is $|A\cdot[C_i]|$, so $\chi(F_i)=|A\cdot[C_i]| \,\chi_{\mathrm{orb}}(\mathcal B_i)$.  So we have
$$
\chi_-(A)=\chi_-(F)=\sum_i\chi_-(F_i)=\sum_i|(A\cdot[C_i]) \, \chi_{\mathrm{orb}}(\mathcal B_i)|.
$$
If $A\cdot[C_i]\ne0$ for each $i$, then each $M_i$ has a fibration over $S^1$. Gluing these fibrations for all $i$ together, we get a fibration of $M$ over $S^1$ whose fiber is $F$.
 \end{proof}

Suppose that $K\subset Y$ is a knot with a frame $\lambda$ and meridian $\mu$. Let $C\subset \partial\nu(K)$ be a simple closed curve which is homologous to $p[\mu]+q[\lambda]$. Then we say that $C$ is the {\it$(p,q)$--cable} of $K$ and denote it $C=C(p,q)\circ K$.

\begin{thm}
Suppose that $m,n\ge2$. If $\frac pq\ge m^2n$, then the $(1,n)$--cable of the $(1,m)$--torus knot is a genus-minimizing knot in the homology class $mn$. 
Moreover, if $q\ne m$, then this cable knot is not a simple knot.
\end{thm}
\begin{proof}
Let $K_1\subset U_1$ be the $(1,m)$--torus knot, $K_2\subset U_1$ the $(1,mn)$--torus knot, and $K_0=C(1,n)\circ K_1$.
By Lemma~\ref{lem:TorusCondition}, $K_1$ and $K_2$ are simple knots. 
Let $\mu_i,\lambda_i$ be the meridian and longitude on $\partial\nu(K_i)$, $i=0,1,2$.

First, we consider the rational genus of $K_2$. The manifold $M_2=L(p,q)-\nu(K_2)$ is Seifert fibered over a disk with two singular points of order $mn$ and $p-qmn$. The homology class of the Seifert fiber on $\partial \nu(K_2)$ is $mn[\mu_2]+[\lambda_2]$.
By Lemma~\ref{lem:BoundarySlope}, we know that  the homology class $(mn)^2q[\mu_2]+p[\lambda_2]$ is null-homologous in $M_2$. By Lemma~\ref{lem:GraphNorm}, a minimal genus rational Seifert surface bounding this homology class is the fiber of a fibration of $M_2$, and its norm is given by
\begin{equation}\label{eq:K2Norm}
\Big|pmn-q(mn)^2\Big|\cdot\Big(1-\frac1{mn}-\frac1{p-qmn}\Big).
\end{equation}

Next, we consider the rational genus of $K_0$. The manifold $M=L(p,q)-\nu(K_0)$ is the union of two Seifert fibered spaces $M_0=\nu(K_1)-\nu(K_0)$ and $M_1=L(p,q)-\nu(K_1)$. The base orbifold of the cable space $M_0$ is an annulus with a singular point of order $n$. The homology class of its  Seifert fiber on $\partial \nu(K_0)$ is $n[\mu_0]+[\lambda_0]$. The base orbifold of  $M_1$ is 
a disk with two singular points of order $m$ and $p-qm$. The homology class of its Seifert fiber on $\partial \nu(K_1)$ is $m[\mu_1]+[\lambda_1]$.  Consider a minimal genus rational Seifert surface $F$ bounding the homology class $(mn)^2q[\mu_0]+p[\lambda_0]$.  
The homology class of the intersection of $F$ with $\partial \nu(K_1)$ is $m^2nq[\mu_1]+pn[\lambda_1]$. By Lemma~\ref{lem:GraphNorm}, the norm of $F$ is
\begin{equation}\label{eq:K0Norm}
\Big|pn-q(mn)^2\Big|\cdot\Big(1-\frac1{n}\Big)+\Big|pmn-qm^2n\Big|\cdot\Big(1-\frac1{m}-\frac1{p-qm}\Big).
\end{equation}
Assuming that $\dfrac{p}{q} \ge m^2n$, \eqref{eq:K2Norm} equals \eqref{eq:K0Norm}. 
Since $K_2$ is a genus minimizer, it follows that $K_0$ is, as well.

If $q\ne m$, then  the complement of $K_1$ is not a solid torus, so the complement of $K_0$ contains an essential torus which is not parallel to the boundary. On the other hand, the complement of $K_2$ is a Seifert fibered space over $D^2$ with two singular points.  Therefore, $K_0$ is not isotopic to $K_2$, the simple knot in its homology class.
\end{proof}

\begin{cor}
Suppose that  $m_1,m_2,\dots,m_k\ge2$. Let $T(1,m_1)$ be the $(1,m_1)$--torus knot. If $\frac pq\ge m_1^2\cdots m^2_{k-1} m_k$, then the iterated torus knot $$C(1,m_k)\circ C(1,m_{k-1})\circ\cdots\circ C(1,m_2)\circ T(1,m_1)$$ is a genus-minimizing knot in the homology class $m_1m_2\cdots m_k$. 
\end{cor}
\begin{proof}
The proof proceeds by induction on $k$. The argument in the above theorem shows that the corresponding iterated torus knot is a genus minimizer.
\end{proof}

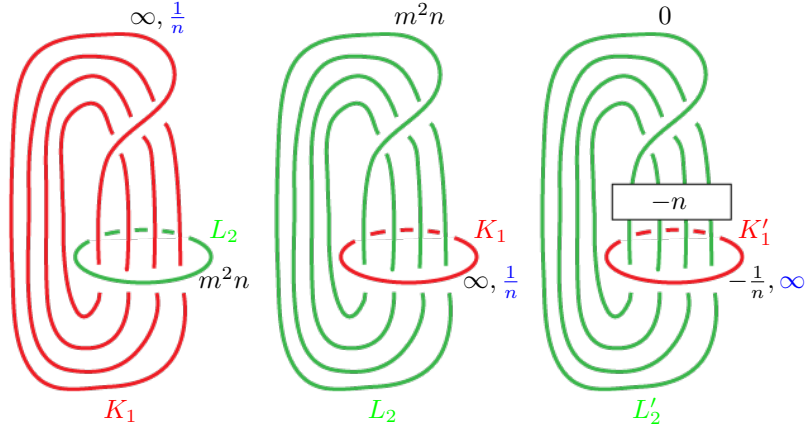
\begin{figure}[ht]
\centering
\includegraphics[height=160pt]{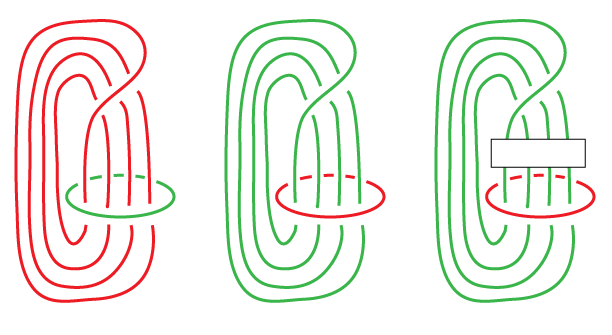}
\put(-214,54){$m^2n$}
\put(-240,154){$\infty,\textcolor{blue}{\frac1n}$}
\put(-250,5){$\textcolor{red}{K_1}$}
\put(-210,73){$\textcolor{green}{L_2}$}
\put(-140,154){$m^2n$}
\put(-114,54){$\infty,\textcolor{blue}{\frac1n}$}
\put(-110,73){$\textcolor{red}{K_1}$}
\put(-150,5){$\textcolor{green}{L_2}$}
\put(-40,154){$0$}
\put(-14,54){$-\frac1n,\textcolor{blue}{\infty}$}
\put(-43,84){$-n$}
\put(-10,73){$\textcolor{red}{K'_1}$}
\put(-50,5){$\textcolor{green}{L'_2}$}
\caption{\label{fig:CableKnot}A Kirby diagram for $K_0$ when $\frac pq=m^2n$. Here we choose $m=4$. On the left and in the middle, $K_0$ is the $(1,n)$--cable of $K_1$. On the right, $K_0$ is the meridian of $K_1'$, and the $-n$ in the box indicates a $-n$ full twist.}
\end{figure}

\begin{rem}
When $p=m^2n$, $q=1$, we get
$$
2g_r(K_1)=\Theta(m)=\frac{mn-n-1}n, \quad 2g_r(K_2)=\Theta(mn)=mn-n-1.
$$
A minimal genus rational Seifert surface for $K_0$ can be visualized explicitly by Kirby calculus.  On the left of Figure~\ref{fig:CableKnot}, we represent the knot $K_0$ by a surgery diagram. Here we put two slopes on $K_1$. One slope $\infty$ indicates the surgery slope which gives the ambient manifold, and the other slope $\frac1n$ indicates the slope of $K_0$ on $\partial\nu(K_1)$. This diagram is identical to the picture in the middle of Figure~\ref{fig:CableKnot}, noting that $K_1 \cup L_2$ is the symmetric torus link $T_{2,2m}$.  The middle picture can then be transformed to the picture on the right of Figure~\ref{fig:CableKnot} by a Rolfsen twist. The knot $L_2'$ is a torus knot $T_{1-mn,m}$, and it has a Seifert surface of genus $(mn-2)(m-1)/2$ that intersects $K_1'$ in $m$ points. This Seifert surface extends to a closed surface in the 0-surgery of $T_{1-mn,m}$. Therefore, the knot $K_0$, which is the meridian of $K_1'$, has a rational Seifert surface $F$ with genus $(mn-2)(m-1)/2$ and $m$ boundary components.  The rational genus of $F$ is precisely the value calculated above for $g_r(K_2) = g_r(K_0)$, so it is a minimal genus rational Seifert surface for $K_0$.
\end{rem}


\section{Torsion curves on the minimal genus nonorientable surface}\label{sect:TorsCurve}

Let $\Pi_h$ denote the closed connected nonorientable surface of genus $h$.  Note that $H_1(\Pi_h)\cong\mathbb Z^{h-1}\oplus\mathbb Z_2$.
The following lemma about the unique torsion class in $H_1(\Pi_h)$ is elementary, and the proof is left to the reader.

\begin{lem}\label{lem:Tors}
The torsion class is the image of the $\mathbb Z_2$--fundamental class under the Bockstein homomorphism $\beta\co H_2(\Pi_h;\mathbb Z_2)\to H_1(\Pi_h;\mathbb Z)$. Moreover, the complement of a simple closed curve $c\subset\Pi_h$ is orientable if and only if $c$ represents the torsion class. \qed
\end{lem}

\begin{defn}
A simple closed curve $c\subset \Pi_h$ is called a {\it torsion curve} if it represents the torsion class.
\end{defn}

When the homology class $a$ has order $2$, the function $\Theta$ is essentially the minimal genus of nonorientable connected surfaces in a given homology class. 

\begin{thm}\label{thm:NOgenus}
Let $Y$ be a rational homology $3$--sphere and $$\beta\co H_2(Y;\mathbb Z_2)\to H_1(Y;\mathbb Z)$$ the Bockstein homomorphism. Given a nonzero class $A\in H_2(Y;\mathbb Z_2)$, let $h(Y,A)$ be the minimal genus of closed, connected, nonorientable surfaces representing $A$.  If $h(Y,A)\ge2$, then we have
$$h(Y,A)=2\Theta(\beta(A))+2.$$
\end{thm}
\begin{proof}
Suppose that $\Pi\subset Y$ is a minimal genus, closed, nonorientable surface representing $A$ and $K\subset \Pi$ is a torsion curve.  Lemma~\ref{lem:Tors} implies that $\Pi - K$ is orientable and $[K]=\beta[\Pi]$. Since $\Pi-K$ is orientable and $\Pi$ is nonorientable, either $\Pi\cap \nu(K)$ is a M\"obius band or $\Pi\cap \nu(K)$ is an annulus and the induced orientations of the two components of $\partial(\Pi\cap \nu(K))$ are coherent on $\partial\nu(K)$. So $\Pi\cap(Y-\nu(K))$ is a rational Seifert surface for $K$ with $|[\mu]\cdot[\partial (\Pi\cap \nu(K))]|=2$, where $\mu\subset\partial\nu(K)$ is the meridian of $K$. Since the genus of $\Pi$ is $h(Y,A)$, $\chi(\Pi\cap(Y-\nu(K)))=\chi(\Pi)=2-h$.
So $\Theta(\beta(A))=g_r(K)\le (h-2)/2$ if $h\ge2$.

On the other hand, if $a$ is in the image of $\beta$, then the order of $a$ is $2$. Let $K$ be a knot representing $a$ such that $g_r(K)=\Theta(a)$, and let $F$ be a minimal genus rational Seifert surface for $K$. It is not hard to see $|[\mu]\cdot[\partial F]|=2$. Depending on whether $|\partial F|=1$ or $2$, we can glue in a M\"obius band or an annulus to $F$ to get a nonorientable surface $\widehat{F}$. Clearly, $\chi(F)=\chi(\widehat F)$. Since $\beta$ is injective, $[\widehat F]\in H_2(Y;\mathbb Z_2)$ is determined by $[K]=\beta[\widehat F]$. Lastly, $h([\widehat F])\le 2g_r(K)+2=2\Theta(a)+2$, which finishes the proof.
\end{proof}

\begin{proof}[Proof of Theorem~\ref{thm:Order2Minimizer}]
If $K$ is a genus minimizer, the proof of Theorem~\ref{thm:NOgenus} shows that $K$ lies on a surface $\widehat F$ such that $\widehat F$ is a minimal genus, closed, nonorientable surface. Rubinstein's result implies that $\widehat F$ is the unique incompressible surface $\Pi$ in $L(2k,q)$ \cite[Theorem 12]{RubinOS}. The proof of Theorem~\ref{thm:NOgenus} also implies that any torsion curve on $\Pi$ is a genus minimizer in the homology class $k$.

When the genus $N(2k,q)$ of $\Pi$ is  at most 3, there is a unique (up to isotopy in $\Pi$) torsion curve on $\Pi$ \cite[Lemma~2.1]{Sch}, so there is a unique genus minimizer in the homology class $k$.
\end{proof}

We record the following construction for the unique incompressible surface $\Pi \subset L(2k,q)$, which has been independently noted by others \cite{Bakercorr,Rubincorr}.  Let $S(2k,q)$ denote the two-bridge link with branched double-cover $L(2k,q)$.  Put $S(2k,q)$ in alternating two-bridge position, let $D$ denote the disk bounded by one of the link components in the projection plane, and perturb the other component to meet $D$ transversely.  Then the preimage of $D$ under the branched covering map $L(2k,q) \to S^3$ is the surface $\Pi$.  It is straightforward to verify that the resulting surface $\Pi$ is non-orientable, and furthermore that its genus equals $N(2k,q)$ (for example, by using the formula for this value that appears in \cite[Theorem 6.1]{BW}).


\section{Stabilizations}\label{sect:Stab}

More examples of genus minimizers can be constructed by stabilizing the cable examples in Section~\ref{sect:Cable}.
For simplicity, we only consider a special case.

We keep the notation from the first paragraph of Section~\ref{sect:Cable}. Let $K_0=C(1,2)\circ C(1,2)\circ L_1$ be a knot in $U_1$. Let $\gamma$ be a small unknot about the ``bottom-left'' crossing between $K_0$ and $L_2$, as in Figure~\ref{fig:BraidStab}. We orient $K_0$, $L_2$, and $\gamma$ so that the linking number between any two of them is positive.  For any integer $k>0$, $(-\frac1k)$--surgery on $\gamma$ transforms $K_0$ into the $k$-fold positive stabilization of the closed braid $K_0$, denoted $K_k$.

\begin{figure}[ht]
\centering
\includegraphics[height=120pt]{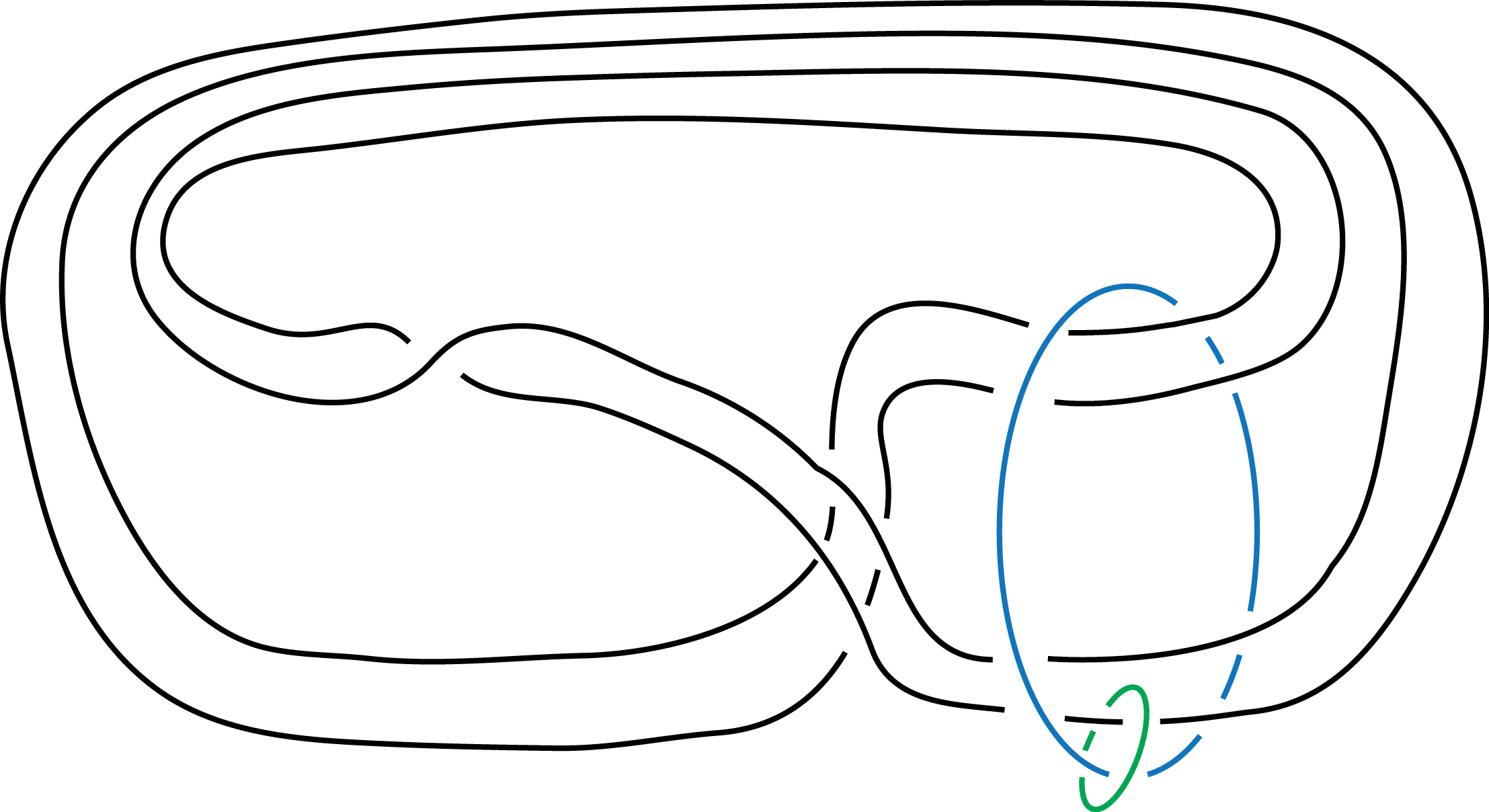}
\put(0,50){$K_0$}
\put(-44,0){$\textcolor{blue}{L_2}$}
\put(-67,-2){$\textcolor{green}{\gamma}$}
\caption{\label{fig:BraidStab}The link $K_0\cup L_2\cup \gamma$. Here $K_0$ is a closed braid with axis $L_2$.}
\end{figure}

There is another way to view the link $K_0\cup L_2\cup \gamma$.
On the left of Figure~\ref{fig:Pi4}, we draw $K_0$ as a curve on a genus $4$ once-punctured non-orientable surface $\Pi_{4,1}\subset U_1$. 
On the right of Figure~\ref{fig:Pi4}, an additional knot $\gamma\subset U_1$ is added to the picture of $K_0\cup L_2$. This 
$3$--component link is isotopic to the link in Figure~\ref{fig:BraidStab}.

\begin{figure}[ht]
\centering
\includegraphics[height=160pt]{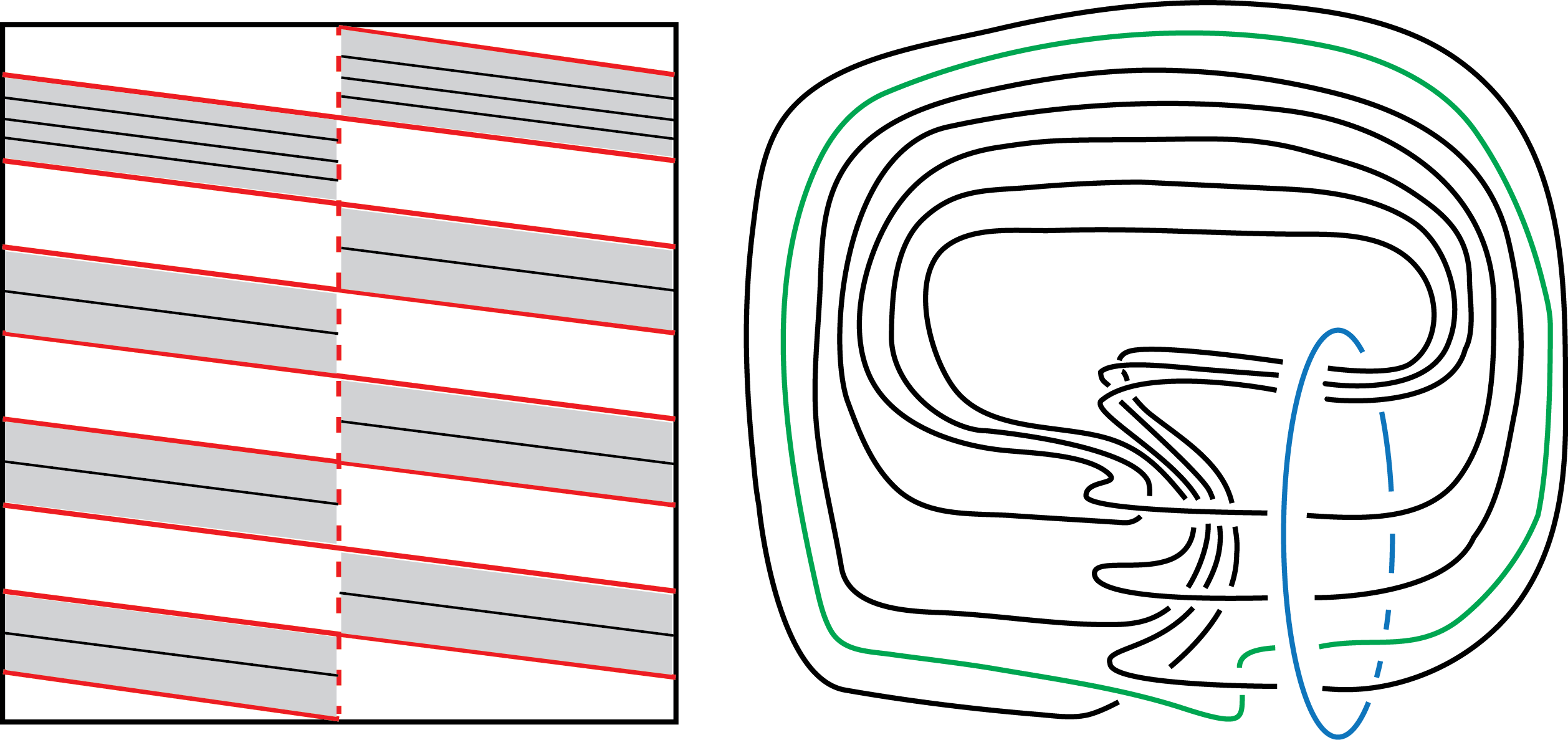}
\put(-264,12){$\scriptstyle1$}
\put(-264,47){$\scriptstyle3$}
\put(-264,85){$\scriptstyle5$}
\put(-272,30){$\scriptstyle2$}
\put(-272,67){$\scriptstyle4$}
\put(-272,104){$\scriptstyle6$}
\put(-272,147){$\scriptstyle1$}
\put(-272,142){$\scriptstyle2$}
\put(-272,137){$\scriptstyle3$}
\put(-264,127){$\scriptstyle4$}
\put(-264,122){$\scriptstyle5$}
\put(-264,117){$\scriptstyle6$}
\put(-43,0){$\textcolor{blue}{L_2}$}
\put(-70,0){$\textcolor{green}{\gamma}$}
\put(0,50){$K_0$}
\caption{\label{fig:Pi4}On the left, we draw a surface $\Pi_{4,1}$ in the solid torus $U_1$. The boundary of $U_1$ is obtained by doubling the square along its horizontal sides, then gluing the two boundary components of the resulting annulus together. The double of the vertical sides of the square become a meridian of $U_1$. Here we think of the new copy (added in the doubling operation) of the square  as lying in the back of the picture. The union of the shaded regions represents the surface, which is a genus $4$ nonorientable surface with one boundary component  of slope $1/8$ on $\partial U_1$. The minimal genus closed non-orientable surface $\Pi\subset L(8,1)$ can be obtained by capping off $\partial\Pi_{4,1}$ with the meridian disk $D$ of $U_2$. There are six embedded arcs in $\Pi_{4,1}$. Joining the ends with same label by disjoint arcs in $D$, we get a torsion curve in $\Pi$ which is isotopic to the image of the knot $K_0$ under the inclusion $U_1\subset L(8,1)$. On the right is the link $K_0\cup L_2\cup \gamma$.}
\end{figure}
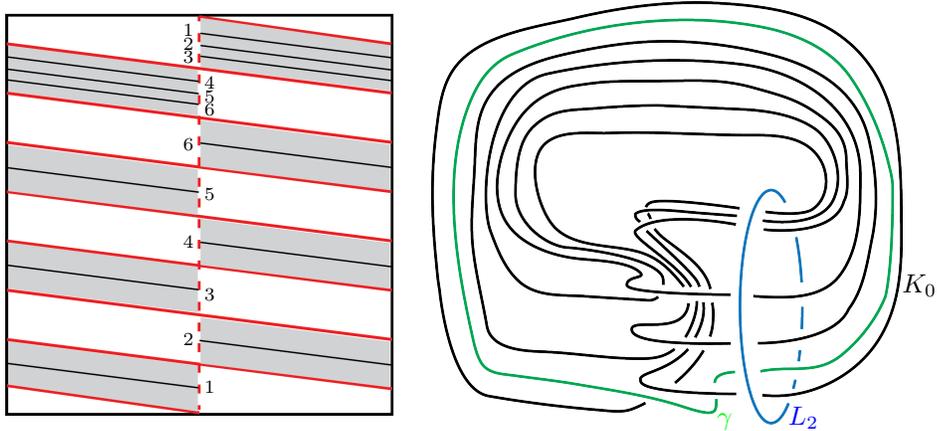

Let $(\mu_0,\lambda_0),(\mathcal M,\mathcal L),(\mu_{\gamma},\lambda_{\gamma})$ be the meridian-longitude pairs on $K_0,L_2,\gamma$.
Let $$M=S^3-\nu(K_0\cup L_2\cup\gamma).$$

The knot $K_0\subset S^3$ is an unknot. Clearly, it bounds a disk $D_0$ which intersects $L_2$ exactly four times and intersects $\gamma$ exactly once. Let $F_0=D_0\cap M$.  Then $\chi_-(F_0)=4$. Moreover, under the map $$\partial\co H_2(M,\partial M)\to H_1(\partial M)=H_1(\partial \nu(K_0\cup L_2\cup\gamma)),$$ we have
$$\partial [F_0]=([\lambda_0],-4[\mathcal M], -[\mu_{\gamma}]).$$

The knot $\gamma$ bounds a disk $D_{\gamma}$ which intersects each of $K_0,L_2$ exactly once. Let $F_{\gamma}=D_{\gamma}\cap M$.  Then $\chi_-(F_{\gamma})=1$ and
$$\partial[F_{\gamma}]=(-[\mu_0],-[\mathcal M],[\lambda_{\gamma}]).$$

From Figure~\ref{fig:Pi4}, we see that the knot $\gamma$ can be isotoped to intersect $\Pi_{4,1}$ exactly once. Let $F=\Pi_{4,1}\cap M$.  Then $F$ is a genus one surface with four boundary components and $\chi_-(F)=4$. We have
$$\partial[F]=(4[\mu_0]+2[\lambda_0],-8[\mathcal M]-[\mathcal L],-[\mu_{\gamma}]).$$

Given relatively prime positive integers $p,q$ with $\frac pq\ge2(k+4)$, consider the homology class
$$(p-2q(k+4))[F_0]+k(p-q(k+4))[F_{\gamma}]+q(k+4)[F],$$
which can be represented by an embedded surface
$F_k$ with $$\chi_-(F_k)=4(p-2q(k+4))+k(p-q(k+4))+4q(k+4)=p(k+4)-q(k+4)^2.$$
We have
\begin{eqnarray*}
\partial[F_k]&=&\big(-kp+q(k+4)^2)[\mu_0]+p[\lambda_0],\\
&&\quad-(p-qk)(k+4)[\mathcal M]-q(k+4)[\mathcal L],\\
&&\quad-(p-q(k+4))[\mu_{\gamma}]+k(p-q(k+4))[\lambda_{\gamma}]\big).
\end{eqnarray*}

The effect of the $(-\frac1k)$--surgery on $\gamma$ is to perform $k$ postive twists along $D_{\gamma}$. Moreover, the slope $\frac{-kp+q(k+4)^2}p$ on $K_0$ becomes $\frac{q(k+4)^2}p$, and the slope $\frac{p-qk}q$ on $L_2$ becomes $\frac pq$. Hence if we also do $\frac {p-qk}q$--surgery on $L_2$, we get a knot $K_k \subset U_1\subset L(p,q)$ in the homology class $k+4$.

The surface $F_k$ has $k+4$ boundary components on $\partial\nu(L_2)$ and $p-q(k+4)$ boundary components on $\partial\nu(\gamma)$. After capping off these boundary components, we get a rational Seifert surface $\widehat{F_k}$ for $K_k\subset L(p,q)$ with  
$$\chi_-(\widehat{F_k})=p(k+4)-q(k+4)^2-(k+4)-(p-q(k+4))=(k+3)(p-q(k+4))-(k+4),$$
and the winding number of $\partial\widehat{F_k}$ on $\nu(K_k)$ is $p$.

On the other hand, similarly to (\ref{eq:K2Norm}), the torus knot $T(1,k+4)\subset L(p,q)$  has a rational Seifert surface which is a fiber of the fibration on $L(p,q)-T(1,k)$ with norm 
$$p(k+3)-q(k+4)(k+3)-(k+4)=\chi_-(\widehat{F_k}),$$
and the winding number of its boundary on $\nu(T(1,k+4))$ is $p$. Moreover, by Lemma~\ref{lem:TorusCondition}, $T(1,k+4)$ is a simple knot, so it is genus-minimizing. Hence $K_k$ is also genus-minimizing. This finishes the proof of the following proposition.

\begin{prop}
When $\frac pq\ge2(k+4)$, the knot $K_k\subset L(p,q)$ is a genus minimizer in the homology class $k+4$. \qed
\end{prop}

Using SnapPy \cite{CDW}, we check that $M$ is a hyperbolic manifold. By the hyperbolic Dehn surgery theorem \cite{ThNotes}, $K_k$ is a hyperbolic knot in $L(p,q)$ for a generic triple $(p,q,k)$. Therefore, $K_k$ is generally not isotopic to $T(1,k+4)$.


\section{Infinitely many homologous genus minimizers}\label{sect:InfiniteMany}

In this section, we establish the second statement in Theorem \ref{thm:NonUnique} by studying the knots that arise from torsion curves on the unique incompressible surface $\Pi \subset L(2k,1)$.

We begin by recalling Osoinach's annulus twist construction \cite{Osoinach}.  Let $Y$ be a closed, oriented 3-manifold and $A \subset Y$ an embedded annulus with boundary $\partial A = \alpha_0 \cup \alpha_1$.  The boundary component $\alpha_i$ has a meridian $\mu_i \subset \partial \nu(\alpha_i)$ and a parallel push-off $\lambda_i = \partial \nu(\alpha_i) \cap  A$, and we orient these curves so that $\mu_i \cdot \lambda_i = +1$.  Osoinach observes that surgery on the link $\alpha_0 \cup \alpha_1 \subset Y$, with framing $\mu_i + \frac{(-1)^i}n \lambda_i$ on $\alpha_i$, results in a 3-manifold $Y_n$ homeomorphic to $Y$ for all $n \in \bZ$.  In fact, the homeomorphism $Y_n \cong Y$ can be taken to be the identity outside of a regular neighborhood of $A$.  We refer to this homeomorphism as the result of doing $n$ {\em twists along} $A$.

Now suppose that $\Pi \subset Y$ is a closed, embedded surface and $\alpha \subset \Pi$ an embedded curve with an orientable (that is, annular) regular neighborhood $N \subset \Pi$.  Then $N$ is co-orientable, and we can push $\alpha$ off of $N$ in the two normal directions to obtain a pair of oriented curves that cobound an annulus $A$ meeting $\Pi$ precisely in $\alpha$.  We can therefore use $A$ to perform the annulus twist construction.  An annular twist of $Y$ along $A$ takes $\Pi$ to itself, and restricting the homeomorphism to $\Pi$ has the effect of doing a Dehn twist along $\alpha$.

We apply this construction in the following way.  Take a disk $D \subset S^3$ and attach $k = a+b+2$ positive half-twisted bands to its boundary, where $a,b \in \bZ^+$.  The result is a non-orientable surface $F \subset S^3$ such that $b_1(F) = k$, $\del F$ is an unknot, and $F$ induces the framing $2k$ on $\del F$.  Choose a curve $\gamma \subset F$  that runs once across each half-twisted band and a curve  $\alpha \subset F$ that runs once across each of two of the bands that are $a+1$ apart and zero times across every other band.  Note that $\gamma$ is uniquely determined up to isotopy within $F$ and so is $\alpha$ once we fix the specific pair of bands that it crosses.  We position $\alpha$ and $\gamma$ so that they intersect minimally; since $a,b \ge 1$, they meet in two points.  See Figure \ref{fig:L(2k,1)}.

\begin{figure}[h]
\centering
\includegraphics[height=2in]{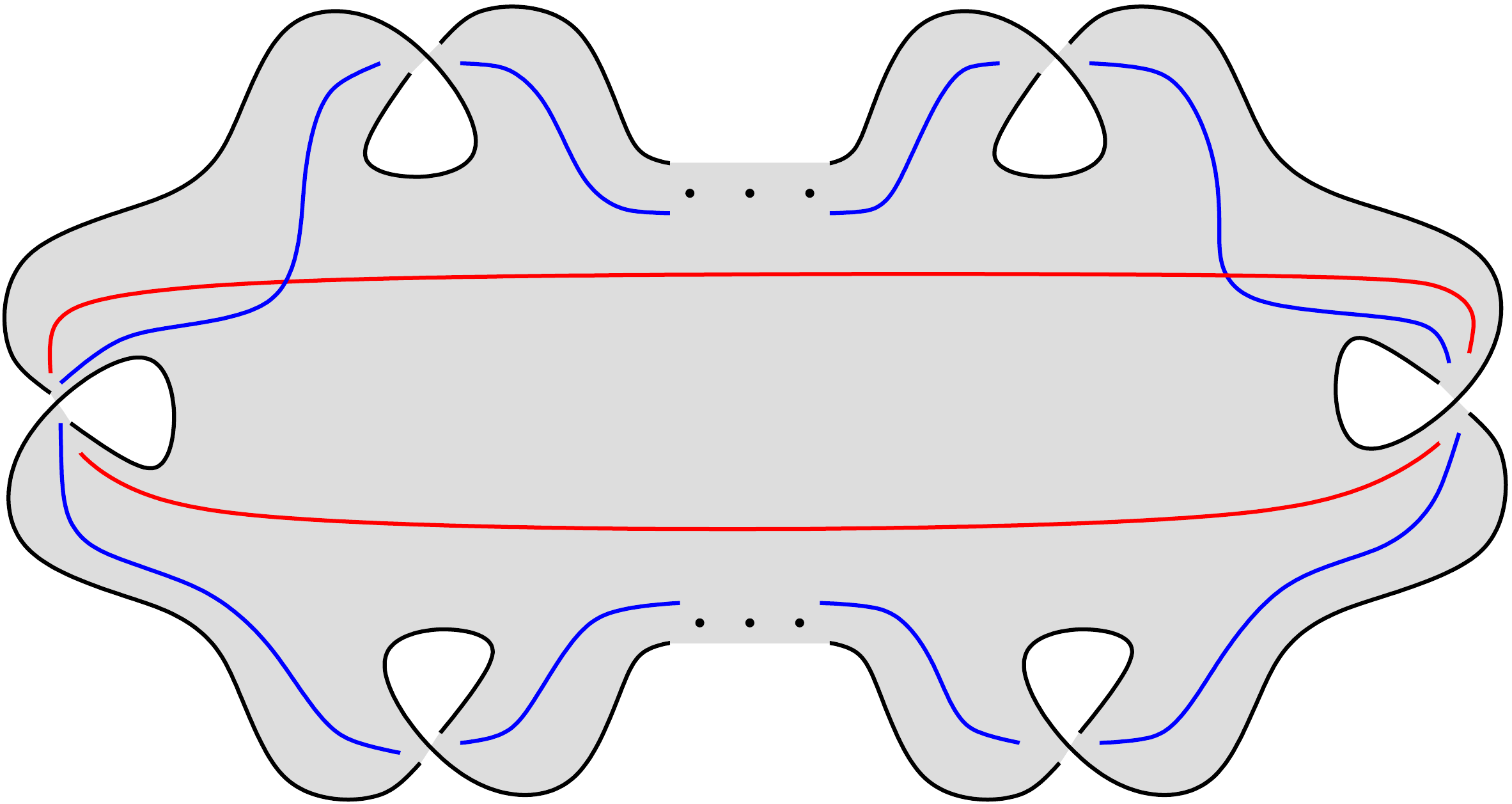}
\put(-105,110){\textcolor{blue}{$\gamma$}}
\put(-70,100){{\textcolor{red}{$\alpha$}}}
\put(-30,120){$2k$}
\caption{\label{fig:L(2k,1)}The surface $F \subset S^3$ shown here consists of a disk with $a \ge 1$ positively half-twisted bands attached to the top, $b \ge 1$ to the bottom, one to the left, and one to the right.  The value $k$ in the surgery coefficient on $\del F$ equals $a+b+2$.}
\end{figure}

The lens space $L(2k,1)$ results from $2k$--surgery along $\del F$, and we obtain the unique incompressible surface $\Pi_k \subset L(2k,1)$ by gluing $F$ to a meridional disk in the surgery solid torus.  We identify $\gamma$ and $\alpha$ with the induced curves on $\Pi_k$.  Observe that $\gamma$ is a torsion curve in $\Pi_k$ and that a regular neighborhood of $\alpha$ in $\Pi_k$ is orientable.  Apply the annulus twist construction as above.  The image of the curve $\gamma$ under $n$ twists along $A$ is a knot $K(a,b,n) \subset L(2k,1)$.  Equivalently, $K(a,b,n)$ is the image of $\gamma \subset \Pi_k$ following $n$ Dehn twists of $\Pi_k$ along $\alpha$.  Note that from the symmetry of the construction, $K(a,b,n)$ is uniquely determined up to isotopy; that is, it does not depend on the specific pair of bands at distance $a+1$ that $\alpha$ crosses.

\begin{thm}\label{thm:InfiniteExamples}

For all $a,b \in \bZ^+$, $n \in \bZ$, $k=a+b+2$, the knots $K(a,b,n)$ are genus minimizers in the homology class $k \in H_1(L(2k,1))$.  For fixed $a,b \gg 0$, the knots $K(a,b,n)$ represent infinitely many distinct hyperbolic knot types in $L(2k,1)$.

\end{thm}

In order to prove the second statement in Theorem \ref{thm:InfiniteExamples}, we rely on a version of the hyperbolic Dehn surgery theorem due to Neumann and Zagier \cite{NZ}.  To state it, let $M$ be a hyperbolic manifold with cusps $\gamma_1,\dots,\gamma_n$, and choose a framing $(\mu_i,\lambda_i)$ for each $\gamma_i$.  Let $M((p_1,q_1),\dots,(p_n,q_n))$ be the result of Dehn filling cusp $\gamma_i$ by $p_i \mu_i + q_i \lambda_i$ for $i=1,\dots,n$, where $p_i,q_i$ are a pair of coprime integers; if no filling is done along $\gamma_i$, then replace the pair by the symbol $\infty$.

\begin{thm}\label{thm:NZ}
The manifold $M((p_1,q_1),\dots,(p_n,q_n))$ is hyperbolic for $\min\{p_i^2 + q_i^2\} \gg 0$.  Its volume is strictly less than that of $M$ if at least one cusp gets filled, and as $\min\{p_i^2 + q_i^2\} \to \infty$, the volume of $M((p_1,q_1),\dots,(p_n,q_n))$ tends to that of $M$.
\qed
\end{thm}

\begin{proof}[Proof of Theorem \ref{thm:InfiniteExamples}]

Since $\gamma \subset \Pi$ is a torsion curve and a Dehn twist of $\Pi$ along $\alpha$ is a homeomorphism of $\Pi$, the image of $\gamma$ under iterated Dehn twists along $\alpha$ is again a torsion curve on $\Pi$.  Therefore, the first assertion in the theorem follows from Theorem \ref{thm:Order2Minimizer}.

To establish the second part, we first describe the knot $K(a,b,n)$ using Kirby calculus.  We push $\alpha$ off in the two normal directions to obtain curves $\alpha_i$ which we frame by $\mu_i + (-1)^i n \lambda_i$, $i=0,1$, in order to apply the annulus twist construction described above.  We then replace the $a$ full twists between $\del F$ and $\gamma$ by a $(-\frac 1a)$--framed curve encircling the two, and similarly for the $b$ full twists between them.  Doing so alters the surgery coefficient on $\del F$ to $2k-a-b = k+2$.  The result of this construction is the surgery diagram for $K(a,b,n)$ displayed in Figure~\ref{fig:surgerylink}.  Note that $\lambda_i = \mu_i + \lambda_{i,S}$, where $\lambda_{i,S}$ is the standard Seifert framing for the curve representing $\alpha_i$.  Therefore, $\mu_i + (-1)^i n \lambda_i = \big((-1)^in+1\big) \mu_i + (-1)^in \lambda_{i,S}$, so we frame $\alpha_i$ by $1 + (-1)^i/n$ as shown.  (Note as well the sign error on $k$ occurring in \cite[Fig. 3]{Osoinach}.)

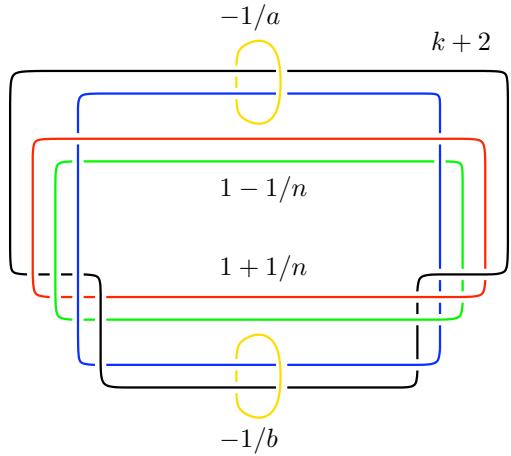
\begin{figure}[h]
\centering
\includegraphics[height=2in]{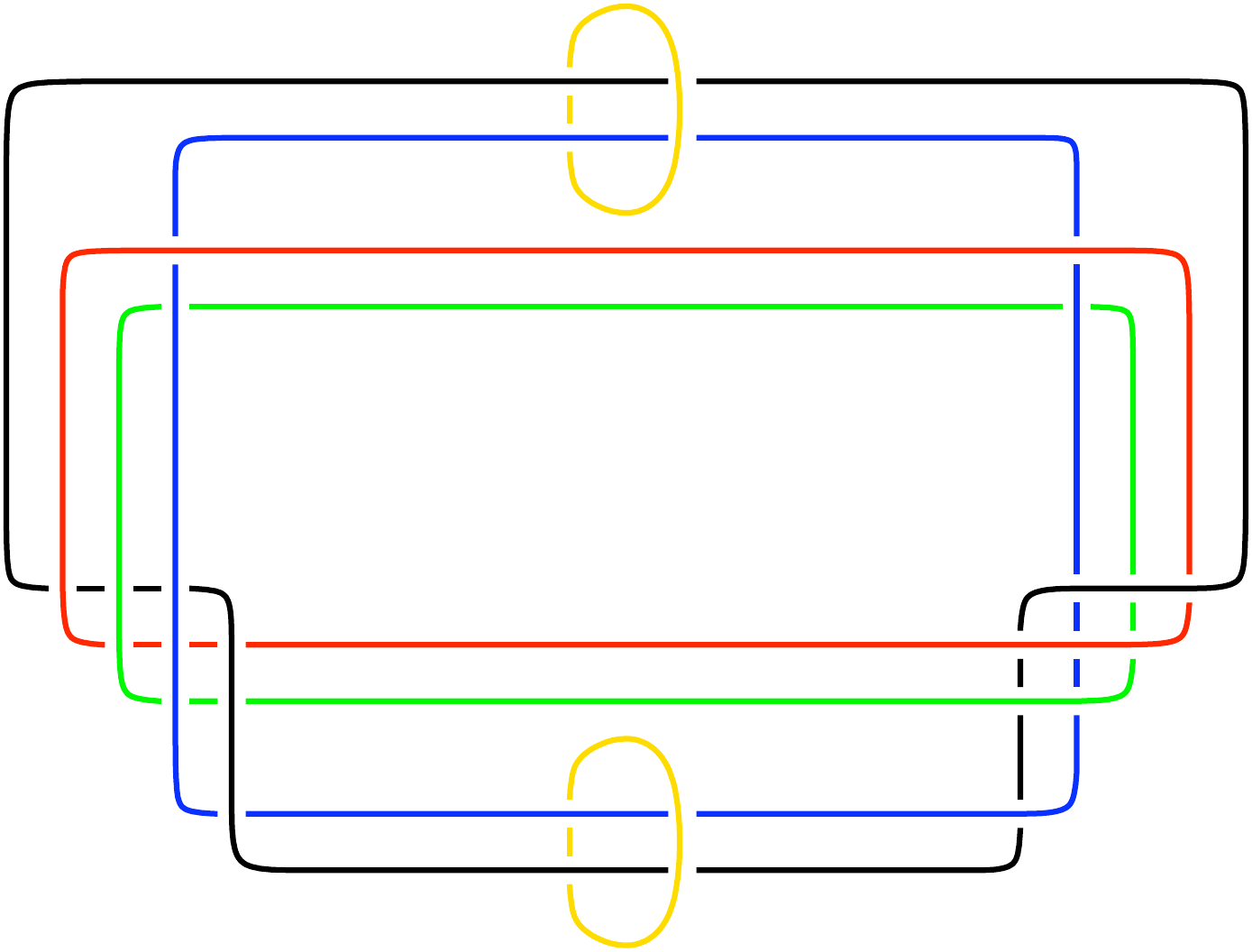}
\put(-110,150){$-1/a$}
\put(-110,-10){$-1/b$}
\put(-110,85){$1-1/n$}
\put(-110,55){$1+1/n$}
\put(-30,140){$k+2$}
\caption{\label{fig:surgerylink}Surgery on the framed link components produces the space $L(2k,1)$, $k = a+b+2$, and the image of the unframed component is the knot $K(a,b,n)$.}
\end{figure}

Let $L \subset S^3$ be the unframed link appearing in Figure \ref{fig:surgerylink} and $M$ its complement.  Label its components $\gamma_1,\dots,\gamma_6$ so that the surgered manifold has the description $M((-1,a),(-1,b),(k+2,1),(n-1,n),(n+1,n),\infty)$.  Let
$$
M_{a,b}=M((-1,a),(-1,b),(k+2,1),\infty,\infty,\infty).
$$
According to SnapPy, $M$ is hyperbolic, so by Theorem \ref{thm:NZ}, the manifold $M_{a,b}$ is hyperbolic for all $a,b \gg 0$.  Fix such a pair $a,b$.  By another application of Theorem \ref{thm:NZ}, there exists a sequence $\{n_i\}$ such that the volumes of the manifolds
$$
M((-1,a),(-1,b),(k+2,1),(n_i-1,n_i),(n_i+1,n_i),\infty)
$$
form a strictly increasing sequence in $i$ with limiting value equal to the volume of $M_{a,b}$.  In particular, all these manifolds are distinct, which is to say that the knots $K(a,b,n_i) \subset L(2k,1)$ are distinguished by their complements.  This concludes the proof of the theorem.
\end{proof}

For fixed values $a,b$, we may use SnapPy to verify that $M_{a,b}$ is hyperbolic.  This is the case, for instance, for $a=1$ and $3 \le b \le 10$.  (For $b=1,2$, SnapPy was inconclusive.)  It follows that the knots $K(1,k-3,n)$ represent infinitely many distinct genus minimizers in the homology class $k \in L(2k,1)$ for $6 \le k \le 13$.  Presumably this is the case for all $k \ge 4$.  Lastly, we note that the simplest interesting example given by this construction is $K(1,1,1)$, which is the knot $K_0$ of Figure \ref{fig:FirstExample}.


\end{document}